\documentclass[12pt]{amsart}
\usepackage[utf8]{inputenc}
\usepackage{mathpazo, url, amssymb, enumerate, todonotes, graphicx} 
\newtheorem{theorem}{Theorem}[section]
\newtheorem{lemma}[theorem]{Lemma}
\newtheorem{proposition}[theorem]{Proposition}

\newtheorem{corollary}[theorem]{Corollary}
\theoremstyle{definition}
\newtheorem{definition}[theorem]{Definition}

\newtheorem{question}[theorem]{Question}
\theoremstyle{remark}

\numberwithin{equation}{section}

\DeclareMathOperator{\re}{Re}
\DeclareMathOperator{\im}{Im}

\title[The number of self-intersections of a trigonometric curve]{A sharp bound on the number of self-intersections of a trigonometric curve}

\author{Sergei Kalmykov}
\address{School of mathematical sciences,   CMA-Shanghai, Shanghai Jiao Tong University, 800 Dongchuan RD, Shanghai 200240, China} 
\email{kalmykovsergei@sjtu.edu.cn} 

\author{Leonid V. Kovalev}
\address{215 Carnegie, Mathematics Department, Syracuse University, Syracuse, NY 13244, USA}
\email{lvkovale@syr.edu}


\begin{document}
\begin{abstract} We obtain a sharp bound on the number of self-intersections of a closed planar curve with trigonometric parameterization. 
Moreover, we show that a generic curve of this form is normal in the sense of Whitney. 
\end{abstract}

\maketitle
\baselineskip6.6mm

\section{Introduction}

A nonzero Laurent polynomial with complex coefficients has the form 
\begin{equation}\label{Lpoly}
p(z) = \sum_{k=m}^n a_k z^k
\end{equation}
where $m, n\in \mathbb Z$, $m\le n$, $a_m\ne 0$, and $a_n\ne 0$. The set of all such polynomials  will be denoted $\mathcal P_{m, n}$. A planar \emph{trigonometric curve} is the image of the unit circle $\mathbb T = \{z\in \mathbb C\colon |z|=1\}$ under a Laurent polynomial, that is, the closed parametric curve $ \{t\mapsto p(e^{it})$, $0\le t\le 2\pi\}$. Such curves arise in many contexts, see~\cite{HongSchicho} for a brief overview. We are interested in the number of their self-intersections. By definition, a \textit{self-intersection of $p$ on $\mathbb T$} is a two-point subset $\{z_1, z_2\}\subset \mathbb T$ where $z_1\ne z_2$ and $p(z_1)=p(z_2)$. Note that if $p$ attains the same value at $s>2$ distinct points on $\mathbb T$, that counts as $s(s-1)/2$ self-intersections. Our main result is the following theorem. 

\begin{theorem}\label{self-intersection-thm} If $n \ge |m| \ge 1$, the number of self-intersections of the Laurent polynomial~\eqref{Lpoly} on $\mathbb T$ is at most 
\begin{equation}\label{upper-bound-thm}
\sigma(m, n) := (n-1)(n-m) - \gcd(n, m) + 1 
\end{equation}
with the following exceptions: (a) $p$ can be written as $q(z^j)$ for some Laurent polynomial $q$ and some integer $j\ne -1, 1$; (b) $n=-m$ and $|a_n| =  |a_m|$. 

The upper bound~\eqref{upper-bound-thm} is attained for each $(m, n)$ such that $n \ge |m| \ge 1$ and $n\ne m$.
\end{theorem}

Since $\overline{p}\in \mathcal P_{-n, -m}$ has the same number of self-intersections as $p$, there is no loss of generality in assuming $|m|\le n$ in Theorem~\ref{upper-bound-thm}. Also, the case of $p\in \mathcal P_{0, n}$ reduces to $p\in \mathcal P_{1, n}$ by disregarding the constant term. 

Theorem~\ref{self-intersection-thm} improves the upper bound in~\cite[Theorem 2.1]{PAMS} which was only sharp in the case $m<0$ and $\gcd(m, n)=1$. Families of curves realizing the upper bound in this case were recently considered by Fomin and Shustin  in~\cite[Proposition 7.15]{FominShustin}. The case $m=-n$ was addressed by Ishikawa in~\cite[Lemma 1.3]{Ishikawa}. The case of $p$ being an algebraic polynomial of degree $n$ was studied by Quine in~\cite{Quine73}. Quine's bound $(n-1)^2$ follows from Theorem~\ref{self-intersection-thm} because after disregarding the constant term in $p$ we have $p\in \mathcal P_{m, n}$ with $m\ge 1$, and 
$\sigma(m, n)\le (n-1)^2$. Beyond the number of self-intersections, the algebraic and geometric properties of the sets $p(\mathbb T)$ with $p\in \mathcal P_{m, n}$ have been studied in~\cite{Fernando}, \cite{HongSchicho}, and~\cite{KovalevYang}. 

The exceptional case (a) of Theorem~\ref{self-intersection-thm} consists precisely of the polynomials that trace the same curve multiple times. Case (b) includes the polynomials that trace a non-closed curve back and forth, for example $z+z^{-1}$. But it also includes some polynomials like $p(z)=(z+2-z^{-1})^4$ where the set $p(\mathbb T)$ is a simple closed curve. A more complicated example of this type is given in~\cite[Example 8]{HongSchicho}. Theorem~2.1 in~\cite{HongSchicho} relates the back-and-forth behavior to the existence of a polynomial parameterization of $p(\mathbb T)$ that is $1-1$ with finitely many exceptions. 

When $m<n$, the exceptional cases of Theorem~\ref{self-intersection-thm} form a nowhere dense closed subset of $\mathcal P_{m, n}$, under the standard topology of a finite-dimensional vector space. Thus, the upper bound on self-intersections applies to \emph{generic} polynomials, those forming an open dense subset of $\mathcal P_{m, n}$.    
In 
Section~\ref{sec:rotation} we show (Theorem~\ref{thm-normal-dense}) that for a generic polynomial $p\in \mathcal P_{m, n}$ the parametric curve $t\mapsto p(e^{it})$ is \emph{normal} in the sense of Whitney~\cite{Whitney1937}, which makes it possible to assign a sign to each self-intersection and relate the signed count to the rotation number (Whitney index) of the curve. 

\section{Preliminaries}\label{sec-prelim}

For $n\in \mathbb Z$ the Chebyshev polynomial $U_{n-1}$ (of second kind) is defined by the identity $U_{n-1}(\cos \theta) = \sin(n\theta)/\sin \theta$.  
The degree of $U_{n-1}$ is $|n|-1$ and the supremum of $|U_{n-1}|$ on $[-1, 1]$ is equal to $|n|$. 

\begin{lemma}\label{U-gcd} \cite[Theorem 4]{RayesTrevisanWang} The greatest common divisor of the polynomials $U_{n-1}$ and $U_{m-1}$ is $U_{d-1}$ where $d=\gcd(n, m)$. Therefore, $U_{n-1}$ and $U_{m-1}$ have exactly $d-1$ common zeros.
\end{lemma}

\begin{lemma}\label{g-lemma} \cite[Lemma 2.3]{PAMS} Consider a Laurent polynomial~\eqref{Lpoly} with $-n<m\le n$, $m\ne 0$, $a_n\ne 0$, and $a_m\ne 0$. Let 
\begin{equation}\label{g-note}
g(t, z) = \sum_{k=m}^{n} a_k U_{k-1}(t) z^{k-m},
\end{equation}
\begin{equation}\label{gs-note}
g^*(t, z) = z^{n-m} \overline{g(\bar t, 1/\bar z)} = \sum_{k=m}^{n} \overline{a_k} U_{k-1}(t) z^{n-k}.
\end{equation}
Then, with $t=\cos \theta$, we have
\begin{equation}\label{g-notation}
g(t, z) = z^{-m} \frac{p(e^{i\theta} z) - p(e^{-i\theta} z)}{e^{i\theta} - e^{-i\theta}}.
\end{equation}
Also, $g$ is a polynomial in $t, z$ of total degree $2n-m-1$, and $g^*$ is a polynomial in $t, z$ of total degree 
$n - m + |m| - 1$. 
\end{lemma}

\begin{lemma}\label{self-intersections} \cite[Lemma 2.4]{PAMS} Let $p$, $g$, $g^*$ be as in Lemma~\ref{g-lemma}. Given a self-intersection of $p_{|\mathbb T}$, write it in the form $\{ e^{i\theta}z, e^{-i\theta}z \}$ where $z\in \mathbb T$ and $e^{i\theta}\in \mathbb T\setminus \{-1, 1\}$. Let $t=\cos \theta$. Then $g(t, z)=g^*(t, z) = g(-t, -z) = g^*(-t, -z)=0$, i.e., the algebraic curves $g=0$ and $g^*=0$ intersect at the points $(t, z)$ and $(-t, -z)$. Different self-intersections correspond to different pairs $\{(t, z), (-t, -z)\}$.  
\end{lemma}

The following bound on intersection multiplicity was proved by Byd\v{z}ovsk\'{y}~\cite{Bydzovsky} and can be found in more recent publications~\cite[Theorem 1.1]{BodaSchenzel}, \cite[Remark 5.4]{Khadam}, and~\cite[Theorem 1]{BosakovaChalmoviansky}. 

\begin{lemma}\label{lemma-with-tangents} Suppose that two polynomials $f$ and $g$ in $\mathbb C[x, y]$ have no common factors and vanish at a point $P\in \mathbb C^2$ with multiplicities $a$ and $b$. If the algebraic curves $f=0$ and $g=0$ have $c$ common tangents at $P$, counted with multiplicities, then their intersection multiplicity at $P$ is at least $ab+c$.     
\end{lemma}

\section{Proof of Theorem~\ref{self-intersection-thm}: upper bound}

First consider the case $1\le |m| < n$. Let $d=\gcd(m, n)$. 
Let $g$ and $g^*$ be as in~\eqref{g-note} and~\eqref{gs-note}. By virtue of Lemma~\ref{self-intersections}, our goal is to show that $g$ and $g^*$ have at most $2\sigma(m, n)$ common zeros $(t, z)$ with $|z|=1$. The upper bound provided by Bezout's theorem (which applies since $g$ and $g^*$ have no common factors~\cite[p. 550]{PAMS}) is 
\[
\deg g \deg g^* = (2n-m-1)(n-m+|m|-1)
\] 
which is generally greater than $2\sigma(m, n)$. However, we will see that the algebraic curves $g=0$ and $g^*=0$ have intersections along the line $z=0$ and also along the line at infinity. 

Indeed, since $g(t, 0) = a_mU_{m-1}(t)$ and $g^*(t, 0)=\overline{a_n}U_{n-1}(t)$, Lemma~\ref{U-gcd} shows that $g$ and $g^*$ have $d-1$ common zeros along the line $z=0$. 

To study the intersections along the line at infinity in $\mathbb CP^2$, we introduce the corresponding homogeneous polynomials: 
\begin{equation}\label{G-homog}
G(t, z, w) = 
\sum_{k=m}^n a_k U_{k-1}(t/w) z^{k-m} w^{2n-k-1}
\end{equation}
and 
\begin{equation}\label{G*-homog}
G^*(t, z, w) = \sum_{k=m}^n \overline{a_k} U_{k-1}(t/w) z^{n-k} w^{k-m+|m|-1}.
\end{equation}

The polynomial $G$ has total degree $2n-m-1$ while its degree with respect to $z$ is $n-m$. Hence $G$ has a zero of order $n-1$ at the point $(t, z, w) = (0, 1, 0)$. More precisely, the polynomial $H(t, w):=G(t, 1, w)$ can be written as the sum of homogeneous polynomials $H_j$ as follows: 
\begin{equation}\label{eq-G-mboth}
H(t, w) = \sum_{j=n-1}^{2n-m-1} H_j (t, w) = 
a_n U_{n-1}(t/w) w^{n-1} + \sum_{j=n}^{2n-m-1} H_j (t, w)
\end{equation}
where each $H_j$ is homogeneous of degree $j$. 

Similarly, $G^*$ has total degree $n-m+|m|-1$ while its degree with respect to $z$ is $n-m$. Expanding
$H^*(t, w):=G^*(t, 1, w)$  into a sum of homogenous terms, we find that $G^*$ has a zero of order $|m|-1$ at $(0, 1, 0)$: 
\begin{equation}\label{eq-G*-mboth}
H^*(t, w) = \sum_{j=|m|-1}^{n - m + |m|-1} H_j^* (t, w) = 
\overline{a_m} U_{m-1}(t/w) w^{|m|-1} + 
\sum_{j=|m|}^{n - m + |m|-1} H_j^* (t, w). 
\end{equation}
The comparison of~\eqref{eq-G-mboth}  and~\eqref{eq-G*-mboth} shows that the algebraic curves $H=0$ and $H^*=0$ have a common tangent line $t/w = r$ whenever $U_{n-1}(r)=U_{m-1}(r)=0$. By Lemma~\ref{U-gcd} there are $(d-1)$ common tangents, which contribute to the intersection multiplicity according to Lemma~\ref{lemma-with-tangents}: 
\begin{equation}\label{eq-I-bound-1}
I_{(0, 1, 0)}(G, G^*) \ge (n-1)(|m|-1) + d-1.    
\end{equation}

If $m>0$, we can already conclude that the number of points $(t, z)$ where $|z|=1$ and both $g$ and $g^*$ vanish is at most  
\[
\begin{split}
& \deg G \deg G^* - I_{(0, 1, 0)}(G, G^*) - (d-1) \\ 
& \le  (2n-m-1)(n-1) - (n-1)(m-1) - 2(d - 1) \\ & = 
2\sigma(m, n). 
\end{split}
\]

If $m<0$, there is also an intersection at $(1, 0, 0)$ due to 
\begin{equation}\label{eq-G-mneg-2}
G(1, z, w) = 2^{n-1} a_n z^{n-m} + O((|z|+|w|)^{n-m+1}),\quad |z|+|w|\to 0   
\end{equation}
and 
\begin{equation}\label{eq-G*-mneg-2}
G^*(1, z, w) = 2^{n-1}\overline{a_n} w^{-2m} + O((|z|+|w|)^{-2m+1}),\quad |z|+|w|\to 0.   
\end{equation}
Hence $I_{(1, 0, 0)}(G, G^*) \ge -2m(n-m)$ and we conclude that the number of points $(t, z)$ where $|z|=1$ and both $g$ and $g^*$ vanish is at most 
\[
\begin{split}
& \deg G \deg G^* - I_{(0, 1, 0)}(G, G^*) - I_{(1, 0, 0)}(G, G^*) - (d-1) \\ 
& \le  (2n-m-1)(n-2m-1) - (n-1)(-m-1)
+ 2m(n-m) - 2(d - 1) \\ & = 
2\sigma(m, n). 
\end{split}
\]

It remains to consider the case $m=-n$. We use the post-composition with an $\mathbb R$-linear function as in the proof of Theorem~2.1 in~\cite{PAMS}. That is, for $|z|=1$ we have $\overline{p(z)}=\sum_{k=-n}^n \overline{a_{-k}} z^k$ and therefore the terms with $z^{-n}$ cancel out in the sum 
\[
\widehat{p}(z):= p(z) - \frac{a_{-n}}{\overline{a_n}} \overline{p(z)}
=\sum_{k=1-n}^n \left(a_k - \frac{a_{-n}}{\overline{a_n}}\overline{a_{-k}} \right) z^k.     
\]
The map $w\mapsto w -  \frac{a_{-n}}{\overline{a_n}}\overline{w}$ is invertible because $|a_{-n}|\ne |a_n|$. Hence the image of $\mathbb T$ under~$p$ has the same number of self-intersections as its image under~$\widehat{p}$. Since $\widehat{p}\in \mathcal P_{m, n}$ with $m\ge 1-n$, it has at most $\sigma(m, n)$ self-intersections. It is easy to see that the function $\sigma$ in ~\eqref{upper-bound-thm} satisfies $\sigma(m, n)\le \sigma(m-1, n)$ for all $m=1-n, \dots, n$. Therefore $\sigma(m, n) \le \sigma(-n, n)$. 

This completes the proof of the upper bound in Theorem~\ref{self-intersection-thm}. The sharpness of this bound is addressed in the following section.    
 
\section{Proof of Theorem~\ref{self-intersection-thm}: sharpness}
 
\begin{proposition}\label{example-sharpness} Suppose $n, m\in \mathbb Z$ with $n > |m|\ge 1$ and $d=\gcd(n, m)$. Then there exist $\epsilon>\delta>0$ such that the Laurent polynomial $p(z) = z^n + \delta z^{m+1} + \epsilon z^m$ has $\sigma(m, n)$ self-intersections on $\mathbb T$. 
\end{proposition}

The proof is based on the following lemma concerning the winding number of $C^1$ curves. We write $N(\gamma)$ for the winding number of a closed curve $\gamma$ about $0$. 
When $\gamma \colon \mathbb T \to \mathbb C$ is differentiable, we denote its tangent vector by $\gamma' = \frac{d}{ds} \gamma(e^{is})$.
Also let $D(a, r)=\{z\in \mathbb C\colon |z-a|<r\}$.

\begin{lemma}\label{lemma-stability-winding}  Let $\gamma \colon \mathbb T \to \mathbb C$ be a $C^1$ map with a finite zero set $\gamma^{-1}(0) = \{z_1,\dots, z_p\}$. If $\gamma'$ does not vanish at $z_1, \dots, z_p$, then there exists $\epsilon>0$ such that for any two curves $\gamma_k\colon \mathbb T\to\mathbb C\setminus \{0\}$ with $\|\gamma_k-\gamma\|_{C^1} < \epsilon$  ($k=1, 2$)  we have $|N(\gamma_1) - N(\gamma_2)|\le p$.
\end{lemma}

\begin{proof} The assumptions on $\gamma$ imply that there exists $r>0$ such that $\gamma(\mathbb T)\cap D(0, r)$ consists of $p$ arcs that pass through $0$ and are close to a parameterized diameter of $D(0, r)$ in the $C^1$ norm. When $\gamma_1$ and $\gamma_2$ are sufficiently close to $\gamma$ in the $C^1$ norm, their intersections with $D(0, r/2)$ also consist of $m$ nearly diametral arcs. The change of argument along any such arc is close to either $\pi$ or $-\pi$. Summing the changes of $\arg \gamma_1$ and of $\arg\gamma_2$ along $p$ nearly diametral arcs leads to the conclusion.     
\end{proof}

\begin{corollary}\label{corollary-winding-zeros} Let $H\colon [a, b]\times \mathbb T \to \mathbb C$ be a $C^1$-smooth map. For each $t$ let $\gamma_t$ denote the closed curve $\gamma_t(s)=H(t, s)$ and suppose that $\gamma_t$ and its tangent vector $\gamma_t'$ do not vanish simultaneously.  Furthermore, suppose that $\gamma_a$ and $\gamma_b$ do not pass through $0$. Then there exist at least $|N(\gamma_a)-N(\gamma_b)|$ pairs $(t, z)\in (a, b)\times \mathbb T$ such that $H(t, z)=0$.     
\end{corollary}

\begin{proof}  Let $T$ be the set of all values $t\in [a, b]$ such that $\gamma_t$ passes through $0$. It suffices to consider the case when $T$ is finite. The set $[a, b]\setminus T$ consists of intervals on which the function $t\mapsto N(\gamma_t)$ is constant. Applying Lemma~\ref{lemma-stability-winding} to each $\gamma_t$ with $t\in T$ we find that the total variation of $t\mapsto N(\gamma_t)$ on $[a, b]$ is bounded above by the number of pairs $(t, z)\in (a, b)\times \mathbb T$ such that $H(t, z)=0$. 
\end{proof}
 
\begin{proof}[Proof of Proposition~\ref{example-sharpness}] Let $d=\gcd(m, n)$ and note that the function $\sigma$ from~\eqref{upper-bound-thm}  can also be written as 
\[
\sigma(m, n) = (n-d)(n-m) + (d-1) (n-m-1).
\] 

The polynomial $g$ from~\eqref{g-note} takes the form 
\begin{equation}\label{g-trinom}
\begin{split}
  g(t, z) & = U_{n-1}(t) z^n + \delta U_{m}(t)z^{m+1} +\epsilon U_{m-1}(t)z^m  \\
  & = \left( U_{n-1}(t) z^{n-m} + \delta U_{m}(t)z +\epsilon U_{m-1}(t)\right) z^m.
\end{split}
\end{equation}
For each $t \in (-1, 1)$ we consider the closed curve 
\begin{equation}\label{eq-closed-curves}
\gamma_t(z) = U_{n-1}(t) z^{n-m} + \delta U_{m}(t)z +\epsilon U_{m-1}(t), \quad |z|=1.    
\end{equation}
We write $N(t)$ for the winding number of $g_t$ about $0$, leaving $N(t)$ undefined when $\gamma_t$ passes through $0$. 

Observe that the rational function 
\begin{equation}\label{eq-rational-Phi}
\Phi(t): = \left(\frac{U_{n-1}(t)}{U_{m-1}(t)}\right)^2 + 
\left(\frac{U_{m-1}(t)}{U_{m}(t)}\right)^2 
\end{equation}
is bounded away from $0$ when $t\in [-1, 1]$. Indeed, any value $t$ with $\Phi(t)=0$ must have $U_{m-1}(t)=0$ but then $U_{n-1}/U_{m-1}$ cannot vanish at $t$ since the zeros of Chebyshev polynomials are simple. Thus, $\Phi$ has no real zeros.

Suppose that $(t_0, z_0)\in (-1, 1)\times \mathbb T$ is such that $\gamma_{t_0}'(z_0)=0=\gamma_{t_0}(z_0)$. From~\eqref{eq-closed-curves} we have
\begin{equation}\label{eq-closed-curves-deriv}
\gamma_{t_0}'(z) = i (n-m) U_{n-1}(t_0) z_0^{n-m} + i\delta U_{m}(t_0)z_0    
\end{equation}
and the system $\gamma_{t_0}'(z_0)=0=\gamma_{t_0}(z_0)$ simplifies to
\begin{equation}\label{eq-closed-curves-system}
(m-n) U_{n-1}(t_0) z_0^{n-m} = \delta U_{m}(t_0)z_0 
= \frac{m-n}{n-m-1} \epsilon U_{m-1}(t_0).
\end{equation} 
In terms of the rational function~\eqref{eq-rational-Phi}, the system~\eqref{eq-closed-curves-system} yields
\begin{equation}\label{eq-closed-curves-system-2}
\Phi(t_0) = \frac{\epsilon^2}{(n-m-1)^2} 
+ \left(\frac{\delta}{\epsilon}\right)^2 \left(\frac{n-m-1}{m-n}\right)^2.
\end{equation}
By choosing $\epsilon$ and $\delta/\epsilon$ sufficiently small, we can make sure that the right side of~\eqref{eq-closed-curves-system-2} is less than $\inf_{[-1,1]}\Phi$. Thus $\gamma_t'$ and $\gamma_t$ do not vanish simultaneously when $\epsilon$ and $\delta$  are chosen in this way.

Observe the following facts about the curves~\eqref{eq-closed-curves}, which easily follow from the properties of Chebyshev polynomials stated in~\S\ref{sec-prelim}. 
\begin{enumerate}[(a)]
    \item When $U_{n-1}(t) > \delta |m+1| + \epsilon |m|$, we have $N(t)=n-m$. If $\delta+\epsilon$ is sufficiently small, this case holds on most of the interval $(-1, 1)$.  
    \item When $U_{n-1}(t) = 0$ and $\epsilon |U_{m-1}(t)| > \delta |m+1|$, we have $N(t) = 0$. There are $n-d$ such  values of $t$ when $\delta/\epsilon$ is sufficiently small.
    \item When $U_{n-1}(t) = 0 = U_{m-1}(t)$, we have $N(t)= 1$. There are $d-1$ such values of~$t$. 
\end{enumerate}
Applying Corollary~\ref{corollary-winding-zeros} to appropriate subintervals of $(-1, 1)$, we find that there are at least
\[
2 (n-d)(n-m)  + 2 (d-1)(n-m-1) = 2\sigma(m, n)
\]
pairs $(t, z)$ such that $g(t, z)=0$. By Lemma~\ref{self-intersections} this implies the existence of at least $\sigma(m, n)$ self-intersections, matching the upper bound in Theorem~\ref{self-intersection-thm}.   
\end{proof}

\section{Rotation number and the normality of trigonometric curves}\label{sec:rotation} 

Whitney~\cite[Theorem 2]{Whitney1937} showed that the self-intersections of a closed curve largely determine its rotation number, which is defined as follows. 

\begin{definition}\label{def-rotation-number} If $\gamma$ is a $C^1$-smooth closed planar curve with nonvanishing tangent vector $\gamma'$, then the \emph{rotation number} of $\gamma$ is the winding number of $\gamma'$ about $0$.     
\end{definition}

A curve $\gamma$ in Definition~\ref{def-rotation-number} is called \emph{regular}. Following~\cite{Whitney1937}, we say that a regular curve is \emph{normal} if it has a finite number of self-intersections and each of them is a simple transversal crossing: that is, $\gamma$ passes through a point twice with non-collinear tangent vectors. Fix a base point $w_0\in \gamma$ that is not a crossing. When $\gamma$ is traced starting from $w_0$, in a neighborhood of each crossing we can identify the ``first'' and ``second'' arcs of $\gamma$, both of which inherit their orientation from $\gamma$. A crossing $w$ is \emph{positive} if the second arc passes $w$ from left to right when viewed from the first arc. Otherwise $w$ is a negative crossing. 

\begin{theorem} (\cite[Theorem 2]{Whitney1937}, see also~\cite[Theorem 2]{Titus}) \label{thm-whitney} Let $\gamma$ be a normal curve with a base point $w_0$ that is not a crossing. Let $N^+$ and $N^-$ be the numbers of positive and negative crossings, respectively.
Then the rotation number of $\gamma$ is equal to $N^+ - N^- + \mu$ where $\mu$ is the sum of the winding numbers of $\gamma$ about two points near $w_0$ and lying on opposite sides of $\gamma$.  
\end{theorem}

The base point $w_0$ in Theorem~\ref{thm-whitney} can be chosen so that $|\mu|=1$, for example if $w_0$ lies in the closure of the unbounded component of the complement of $\gamma$. 

Let us say that a Laurent polynomial $p$ is \emph{normal} if the curve $\gamma(t)=p(e^{it})$ is normal. Recall that 
$\mathcal P_{m, n}$ is the set of all Laurent polynomials of the form~\eqref{Lpoly}. The set of normal polynomials will be denoted $\mathcal N_{m, n}$. Its complement, $\mathcal P_{m, n}\setminus \mathcal N_{m, n}$ is the union of three sets defined as follows:  
\begin{itemize}
    \item $p\in \mathcal E_{m, n}^Z$ if $p'$ vanishes on $\mathbb T$;
    \item $p\in \mathcal E_{m, n}^{NT}$ if $p$ has a non-transversal self-intersection on $\mathbb T$;
    \item $p\in \mathcal E_{m, n}^{MC}$ if $p$ has a multiple crossing on $\mathbb T$, that is, $p$ takes the same value at three (or more) different points of $\mathbb T$. 
\end{itemize}
It is clear that $\mathcal E_{m, n}^Z$ is a closed subset of $\mathcal P_{m, n}$. The other two exceptional sets are not necessarily closed in $\mathcal P_{m, n}$ but they are closed in $\mathcal P_{m, n}\setminus \mathcal E_{m, n}^Z$ because sufficiently small perturbations of regular curves do not create new multiple crossings or violate transversality~\cite[Lemma~1]{Titus}. Consequently, $\mathcal N_{m, n}$ is an open subset of $\mathcal P_{m, n}$. 

Let $\mathcal P_n$ be the set of algebraic polynomials of degree $n$. This is a larger set than $\mathcal P_{0, n}$ because the latter requires $a_0\ne 0$. 

\begin{lemma}\label{lemma-multiple-crossings} If $m < n$, then the complement of the set $\mathcal E_{m, n}^{MC}$ is dense in $\mathcal P_{m, n}$. 
\end{lemma}

\begin{proof} We may assume $n>0$, by replacing $z$ with $1/z$ if necessary. 
Suppose $p\in \mathcal E_{m, n}^{MC}$, that is, there exist $w\in \mathbb C$ and distinct points $\zeta_k\in \mathbb T$ such that $p(\zeta_k)=w$ for $k=1, 2, 3$. We begin by ruling out the case $n = m+1$. Indeed, in this case $p(z)=z^m (a_{n} z + a_m)$ which implies that the points $\zeta_k$, $k=1, 2, 3$, lie on the circle $\{z \colon |a_{n} z + a_m| = |w|\}$ which is not concentric with $\mathbb T$. This is clearly impossible. 

From now on $n\ge m+2$. We consider the cases $m\le 0$ and $m\ge 1$ separately. 

\textbf{Case $m\le 0$.} The polynomial $z^{-m}(p(z)-w)$ belongs to $\mathcal P_{n-m}$ and vanishes at each $\zeta_k$. It follows that $n-m\ge 3$ and there exists 
$q\in \mathcal P_{n-m-3}$ such that 
\begin{equation}\label{eq-factor-zeros}
p(z) = w + (z-\zeta_1)(z-\zeta_2)(z-\zeta_3) z^{m}q(z).    
\end{equation}
Formula~\eqref{eq-factor-zeros} provides a smooth surjection from the space of tuples $(w, \zeta_1, \zeta_2, \zeta_3, q)$, that is $\mathbb C\times \mathbb T^3\times \mathcal P_{n-m-3}$, onto $\mathcal E_{m, n}^{MC}$. The product $\mathbb C\times \mathbb T^3\times \mathcal P_{n-m-3}$ is a real manifold of dimension $2(n-m)+1$, while 
$\dim_{\mathbb R}\mathcal P_{m, n} = 2(n-m)+2$.  
Thus, $\mathcal E_{m, n}^{MC}$ is a measure zero subset of $\mathcal P_{m, n}$.   

\textbf{Case $m\ge 1$.} The Laurent polynomial $z^{-m}(p(z)-w)$ need not be in $\mathcal P_{n-m}$, so we factor $p(z)-w$ instead: 
\begin{equation}\label{eq-factor-zeros-2a}
p(z)-w = (z-\zeta_1)(z-\zeta_2)(z-\zeta_3) \, q(z)
\end{equation}
where $q\in \mathcal P_{n-3}$. 
Since $p(0)=0$, the formula~\eqref{eq-factor-zeros-2a} can be written as 
\begin{equation}\label{eq-factor-zeros-2}
p(z) =(z-\zeta_1)(z-\zeta_2)(z-\zeta_3) \, q(z) + \zeta_1\zeta_2\zeta_3 \, q(0).     
\end{equation}
The formula~\eqref{eq-factor-zeros-2} provides a smooth map from $\mathbb T^3\times \mathcal P_{n-3}$ onto $\mathcal E_{m, n}^{MC}$. If $m=1$, the proof concludes with the observation that 
\[
\dim_{\mathbb R}(\mathbb T^3\times \mathcal P_{n-3}) = 2n-1 < 2n = \dim_{\mathbb R} \mathcal P_{1, n}. 
\]
From now on $m > 1$. In order for $p$ in~\eqref{eq-factor-zeros-2} to belong to $\mathcal P_{m, n}$, the coefficients of $q$ must satisfy a system of $m-1$ homogeneous linear equations. Let us write $(z-\zeta_1)(z-\zeta_2)(z-\zeta_3)=z^3+c_2z^2+c_1z+c_0$ and $q(z)=\sum_{k=0}^{n-3} b_k z^k$. The coefficients of $z^1, \dots, z^{m-1}$ in $p$, which must all be $0$, are given by the matrix-vector product
\begin{equation}\label{eq-matrix-vector}
\begin{pmatrix}
c_1 & c_0 & 0 & 0 & 0 &  \cdots & 0 \\
c_2 & c_1 & c_0 & 0 & 0 & \cdots & 0 \\
1 & c_2 & c_1 & c_0 & 0 & \cdots & 0 \\
0 & 1 & c_2 & c_1 & c_0 & \cdots & 0 \\
\vdots & \vdots & \vdots & \vdots & \vdots & \ddots & \vdots \\
0 & 0 & 0 & 0 & 0 & \cdots & c_0
\end{pmatrix} 
\begin{pmatrix}
b_0 \\ b_1 \\ b_2 \\ b_3 \\  \vdots \\ b_{m-1}
\end{pmatrix}
\end{equation}
Since $c_0 = -\zeta_1\zeta_2\zeta_3\ne 0$, the matrix in~\eqref{eq-matrix-vector} has rank $m-1$. Specifically, $b_1,\dots, b_{m-1}$ can be expressed as linear functions of $b_0$ with coefficients depending smoothly on $\zeta_1,\dots, \zeta_3$. This leaves $n-m-1$ free coefficients in $q$: namely, $b_0$ and $b_m,\dots, b_{n-3}$. Thus $\mathcal E_{m, n}^{MC}$ is the image of an $(2n-2m+1)$-dimensional space under a smooth map, which yields the desired conclusion since $\dim_{\mathbb R}\mathcal P_{m, n} = 2n-2m+2$.  
\end{proof}

\begin{lemma}\label{lemma-change-radius} Suppose $p\in \mathcal P_{m, n}$ is a Laurent polynomial that is not of the form $p(z) = q(z^j)$ with $|j|>1$. 
Then, except for a discrete set of values $r>0$, the rescaled polynomial $p_r(z) = p(r^{-1}z)$ belongs to neither $\mathcal E_{m, n}^Z$ nor $\mathcal E_{m, n}^{NT}$.
\end{lemma}

\begin{proof} Instead of rescaling the polynomial, we will consider the behavior of $p$ on the circles $\mathbb T_r = \{z \colon |z|=r\}$. 
Theorem~\ref{self-intersection-thm} implies that $p$ has finitely many self-intersections on such circles, with at most one exceptional value of $r$ (which arises when $m=-n$). We will focus on the values of $r$ such that $p'\ne 0$ on $\mathbb T_r$ and $p$ has finitely many self-intersections on $\mathbb T_r$. Let $A$ be a circular annulus such that $\mathbb T_r\subset A$ and $p'\ne 0$ on $A$. 

Non-transversal self-intersections of $\gamma_r$ correspond to pairs $(z_1, z_2)\in A^2$ such that $|z_1|=|z_2|=r$, $p(z_1)=p(z_2)$, and the ratio $\frac{z_1 p'(z_1)}{z_2 p'(z_2)}$ is real. Let $E_{NT}$ be the set of all such pairs. We claim that $E_{NT}$ is a discrete subset of $A^2$.

Suppose toward a contradiction that $(z_1, z_2)$ is a non-isolated point of $E_{NT}$. Since $p$ is locally invertible on $A$, there exist neighborhoods $U_k$ of $z_k$ and a conformal map $\varphi := p_{|U_2}^{-1}\circ p_{|U_1}$ such that $\varphi(z_1)=z_2$ and $p(\varphi(z))=p(z)$ for all $z\in U_1$. It follows that $p'(\varphi(z))= p'(z)/\varphi'(z)$. 

Any point of $E_{NT}\cap (U_1\times U_2)$ can be written as $(z, \varphi(z))$. Since 
\[
\frac{z p'(z)}{\varphi(z) p'(\varphi(z))} = \frac{z \varphi'(z)}{\varphi(z)}
\]
it follows that $z_1$ is a limit point of the set 
\[
W_1 := \left\{z\in U_1 \colon |\varphi(z)|=|z| \text{ and } 
\frac{z \varphi'(z)}{\varphi(z)} \in \mathbb R\right\}.
\]
Each of the two equations $|\varphi(z)/z| = 1$ and $\im \frac{z \varphi'(z)}{\varphi(z)} = 0$ describes a set that is locally a finite union of analytic curves. Since their intersection has an accumulation point $z_1$, it must contain some analytic curve $\Gamma$ passing through $z_1$. (Recall that when two analytic curves intersect at infinitely many points, they coincide.) Let $\Gamma'$ be its tangent vector.

Let $\psi(z)=\varphi(z)/z$. Note that $\psi$ is a nonconstant function, for otherwise $p$ would have rotational symmetry contrary to the assumption of the lemma. Since $\psi(\Gamma)\subset \mathbb T$, it follows that
\begin{equation}\label{eq-gamma-imag}
\re \left(\frac{\psi'(\Gamma) \Gamma'}{\psi(\Gamma)}\right) =  0 
\end{equation}
at every point of $\Gamma$. Using the identity $\varphi'(z)=\psi(z)+z\psi'(z)$ we find that
\[
\frac{z \varphi'(z)}{\varphi(z)} = 1 + \frac{z\psi'(z)}{\psi(z)}
\]
and therefore $z\psi'(z)/\psi(z)$ is real-valued on $\Gamma$. This simplifies~\eqref{eq-gamma-imag} to 
$\re (\Gamma'/\Gamma) =  0$ from where it follows that $\Gamma\subset \mathbb T_r$. But then $\phi(\Gamma)\subset \mathbb T_r$ as well, and the identity $p(\phi(z))=p(z)$ produces infinitely many self-intersection of $p$ on $\mathbb T_r$, a contradiction. 

We have proved that, outside of a discrete exceptional set of radii, the restriction of $p$ to $\mathbb T_r$ has nonzero derivative and only transversal self-intersections.   
\end{proof}

\begin{theorem}\label{thm-normal-dense} If $m<n$, then $\mathcal N_{m, n}$ is an open dense subset of $\mathcal P_{m, n}$. 
\end{theorem}

\begin{proof} Only density remains to be proved. Given any polynomial $\mathcal P_{m, n}$, we can perturb it to make sure all the coefficients $a_k$ for $m\le k\le n$ are nonzero, which ensures that the perturbed polynomial $p$ satisfies the assumption of Lemma~\ref{lemma-change-radius}.  
Choose $r$ close to $1$ so that the rescaled polynomial $p_r$ is not in $\mathcal E_{m, n}^Z\cup \mathcal E_{m, n}^{NT}$ which is a closed subset of $\mathcal P_{m, n}$. Lemma~\ref{lemma-multiple-crossings} implies that all polynomials in a sufficiently small neighborhood of $p_r$ are normal, except possibly for $p_r$ itself.  
\end{proof}

The following result generalizes Lemma 1.3 (2) from~\cite{Ishikawa} which concerns the special case $m=-n$. 

\begin{proposition}\label{normal-rotation} For every $p\in \mathcal P_{m, n}\setminus \mathcal E_{m, n}^{Z}$ the rotation number of the curve $\gamma(t)= p(e^{it})$ is between $m$ and $n$, and these bounds are sharp.      
\end{proposition}

\begin{proof} The tangent vector $\gamma'$ can be written in terms of the coefficients~\eqref{Lpoly} as
$\gamma'(t) = \sum_{k=m}^n ika_k e^{ikt}$. By the argument principle, the winding number of $\gamma'$ about $0$ is the sum of orders of the zeros of the rational function $f(z) = \sum_{k=m}^n ika_k z^k$ in the unit disk, where we count poles as zeros of negative order. Clearly, $f$ has a zero of order $m$ at $0$ and it may have up to $n-m$ zeros elsewhere, namely the roots of the polynomial $\sum_{k=m}^n ika_k z^{k-m}$. Thus, the number of zeros of $f$ is between $m$ and $n$. 

The sharpness of both bounds is demonstrated by $p(z)=z^m + \epsilon z^n$ and $p(z)=\epsilon z^m + z^n$ where $\epsilon$ is sufficiently small.  
\end{proof}

Theorem~\ref{thm-whitney} and Proposition~\ref{normal-rotation} show that although the total number of crossings of a polynomial $p\in \mathcal N_{m, n}$ can be of order $\max(m^2, n^2)$ by Theorem~\ref{self-intersection-thm}, with a suitable choice of a base point we have $|N^+ - N^-|\le \max(|m|, |n|)+1$, indicating that the crossings are nearly balanced between positive and negative ones.

\section{Point preimages under generic polynomials}

Lemma~\ref{lemma-multiple-crossings} showed that for a generic polynomial, the image of the unit circle $\mathbb T$ has no multiple crossings. It is natural to ask whether a stronger statement is true. 

\begin{question}\label{question-doomed} Is it true that for a generic polynomial $p\in \mathbb P_{m, n}$  all curves $p(\mathbb T_r)$, $r>0$, have no multiple crossings?
\end{question}

An equivalent way to state the property in Question~\ref{question-doomed} is: for every $w\in \mathbb C$ and every $r>0$ the set 
$p^{-1}(w)\cap \mathbb T_r$ has at most two elements. The following lemma offers some hope of controlling all preimages $p^{-1}(w)$ at once, at least by bounding their cardinality. 

\begin{lemma}  For a generic polynomial $p\in \mathcal P_n$ the preimages $p^{-1}(w)$ have at least $n-1$ elements for every $w\in \mathbb C$.     
\end{lemma}

\begin{proof} The desired property of $p$ is preserved if $p$ is replaced by $\alpha p + \beta$ where $\alpha, \beta\in \mathbb C$ and $\alpha\ne 0$. Therefore, we can normalize $p$ so that $p(0)=0$ and $p'$ is monic. Such polynomials are uniquely determined by their critical points $u_1, \dots, u_{n-1}$ (listed with multiplicity).

Beardon, Carne, and Ng~\cite{BeardonCarneNg} studied the map $\theta\colon \mathbb C^{n-1}\to \mathbb C^{n-1}$ which sends each vector $(u_1, \dots, u_{n-1})\in \mathbb C^{n-1}$ to $(v_1, \dots, v_{n-1})$ where 
\[
v_k = \int_{0}^{u_k} (\zeta-u_1)\cdots (\zeta-u_{n-1})\,d\zeta.
\]
If $u_1, \dots, u_{n-1}$ are the critical points of a polynomial $p$ (with the above normalization) then $v_1, \dots, v_{n-1}$ are its critical values. It is shown in~\cite[p. 348]{BeardonCarneNg} that there is an open dense subset $\mathcal U\subset \mathbb C^{n-1}$ such that $\theta(u)$ has distinct nonzero entries for each $u\in \mathcal U$. Therefore, the critical values of a generic polynomial are distinct. 

If $p$ has $n-1$ distinct critical values, then for every $w\in \mathbb C$ the polynomial $p-w$ has at most one multiple root, in which case that root has multiplicity two. It follows that $p^{-1}(w)$ has at least $n-1$ elements.     
\end{proof}

Despite the above, Question~\ref{question-doomed} has a negative answer. 

\begin{proposition}\label{prop-answer-no}
If $p(z)=\sum_{k=0}^3 a_k z^k$ where $0<|a_1|^2 < |a_2a_3|$, then there exists $w\in \mathbb C$ such that $p^{-1}(w)$ consists of three distinct numbers with the same modulus.     
\end{proposition}

\begin{lemma}\label{lemma-three-numbers} If $|\zeta|<1$, then there exist distinct unimodular numbers $\zeta_1, \zeta_2, \zeta_3$ such that $\zeta_1+\zeta_2+\zeta_3 = \zeta$ and $\zeta_1 \zeta_2 \zeta_3 = 1$.     
\end{lemma}

\begin{proof}
For $0\le r\le 1$ consider the curves $\gamma_r(t) = 2re^{it} + e^{-2it}$, $0\le t\le 2\pi$. 
The winding numbers of $\gamma_0$ and $\gamma_1$ about $\zeta$ are $-2$ and $1$, respectively. Therefore, there exists $r\in (0, 1)$ and $t\in [0, 2\pi]$ such that $\zeta = 2re^{it} + e^{-2it}$. The numbers
\[
\zeta_1 = e^{-2it},\quad 
\zeta_2, \zeta_3 = (r\pm i \sqrt{1-r^2})e^{it}
\]
satisfy $\zeta_1+\zeta_2+\zeta_3 = \zeta$ and $\zeta_1 \zeta_2 \zeta_3 = 1$. Finally, note that if two of three unimodular numbers coincide, their sum is at least $1$ in the absolute value. Hence $\zeta_1, \zeta_2, \zeta_3$ are distinct.
\end{proof}

\begin{proof}[Proof of Proposition~\ref{prop-answer-no}] The claimed property of $p$ is preserved if it is replaced by $\alpha p(\beta z)$ with $\alpha, \beta \ne 0$. The condition $0<|a_1|^2 < |a_2a_3|$ is preserved as well. 

We choose $\alpha$ and $\beta$ so that 
\begin{equation}\label{eq-alpha-beta}
\alpha \beta^3 a_3 = 1 \quad \text{and}  \quad   
\alpha \beta^2 a_2 = - \overline{\alpha \beta a_1}. 
\end{equation}
Indeed, the system~\eqref{eq-alpha-beta} can be easily solved for $|\alpha|$ and $|\beta|$ and separately for $\arg \alpha$ and $\arg \beta$. 
We thus obtain 
\[
\alpha p(\beta z) = z^3 - \zeta z^2 + \bar \zeta z + c
\]
where $|\zeta|<1$. Let $\zeta_1, \zeta_2, \zeta_3$ be the numbers provided by Lemma~\ref{lemma-three-numbers}. It follows that 
\[
\alpha p(\beta z) = (z-\zeta_1)(z-\zeta_2)(z-\zeta_3) + c + 1 
\]
which proves the proposition. 
\end{proof}

\section{Self-intersections of rational functions} 

Theorem~\ref{self-intersection-thm} concerns the images of $\mathbb T$ under Laurent polynomials, which form a narrow subset of rational functions. One may ask about the images of $\mathbb T$ under general rational functions. It turns out that this approach does not uncover interesting phenomena such as the interaction of $m$ and $n$ in the bound~\eqref{upper-bound-thm}. Instead we return to Quine's bound from~\cite{Quine73} but with a less clear picture of the exceptional cases. 

Indeed, consider a rational function 
\begin{equation}\label{eq-rational-polys}
r(z) = \frac{p(z)}{q(z)} \quad \text{ where }    
p(z)=\sum_{k=0}^n a_k z^k \text{ and }  q(z)=\sum_{k=0}^n b_\ell z^\ell.
\end{equation}
Assume that the function $r$ is not constant and is not a Blaschke product.

By applying a suitable M\"obius transformation, we can make sure that the coefficients $a_0, b_0, a_n, b_n$ are nonzero and satisfy $a_n b_0 - a_0 b_n\ne 0$. Indeed, we can choose $\zeta\in \mathbb D$ such that $r(\zeta)$ and $r(1/\bar \zeta)$ are distinct nonzero complex numbers. Consider the M\"obius transformation $\phi(z)=(z+\zeta)/(1+\bar \zeta z)$. 
Replacing $r$ with the composition $r\circ \phi$ does not change the image of $\mathbb T$. 
For the new function, $r(0)$ and $r(\infty)$ are distinct elements of $\mathbb C\setminus\{0\}$. 
Since the coefficients in~\eqref{eq-rational-polys} satisfy $a_0/b_0 = r(0)$ and 
$a_n/b_n = r(\infty)$, the claim follows. 

We have 
\[
r(e^{i\theta}z) - r(e^{-i\theta}z)
= \frac{p(e^{i\theta}z)q(e^{-i\theta}z) - p(e^{-i\theta}z)q(e^{i\theta}z)  }{q(e^{i\theta}z) q(e^{-i\theta}z)}
\]
where the numerator has degree at most $2n-1$ and is divisible by $z$. Factoring out $z$, we get the polynomial 
\[
\sum_{k,\ell=0}^n a_k b_\ell \left(e^{i(k-\ell)}-e^{i(\ell-k)}\right) z^{k+\ell-1}.
\]
Dividing by $e^{i\theta}-e^{-i\theta}$ yields
\begin{equation}\label{eq-rational-gpoly}
g(t, z) = 
\sum_{k,\ell=0}^n a_k b_\ell U_{k-\ell-1}(t) z^{k+\ell-1}
\end{equation}
where $t=\cos \theta$ as earlier. The total degree of $g$ is $2n-2$ and its degree with respect to $t$ is $n-1$, because $g$ can be written as 
\[
g(t, z) = 2^{n-1}  (a_n b_0 - a_0 b_n) t^{n-1} z^{n-1} + \text{(lower order terms in $t$)}
\] 
Hence, the homogenization of $g$, denoted $G(t, z, w)$, vanishes at $(t, z, w) = (1, 0, 0)$ with order $n-1$. 

The conjugate-reciprocal polynomial
\begin{equation}\label{eq-rational-gstar}
g^*(t, z) = 
\sum_{k,\ell=0}^n \overline{a_k b_\ell} U_{k-\ell-1}(t) z^{2n-k-\ell-1}
\end{equation}
contains the term 
\[
2^{n-1} (\overline{a_n b_0 - a_0 b_n}) t^{n-1} z^{n-1} 
\]
and all other terms contain $t$ to a power less than $n-1$. Thus, its homogenization $G^*$ also vanishes at $(1, 0, 0)$ with order $n-1$. 

Moreover, both algebraic curves have $z=0$ as a tangent line of multiplicity $n-1$ at $(1, 0, 0)$. This means they have $(n-1)^2$ pairs of coincident tangent lines. By Lemma~\ref{lemma-with-tangents}
\[
I_{(1, 0, 0)}(G, G^*) \ge (n-1)^2 + (n-1)^2
= 2(n-1)^2. 
\]
Since $\deg g \deg g^*=4(n-1)^2$, Bezout's theorem shows that $g=0$ and $g^*=0$ have at most $4(n-1)^2 - 2(n-1)^2 = 2(n-1)^2$ common points in $\mathbb C^2$, unless they have a common component. Using Lemma~\ref{self-intersections} we arrive at the following conclusion. 

\begin{proposition} The rational function~\eqref{eq-rational-polys} has at most $(n-1)^2$ self-intersections on $\mathbb T$ unless the polynomials $g$ and $g^*$ from~\eqref{eq-rational-gpoly} and~\eqref{eq-rational-gstar} have a nontrivial common factor.     
\end{proposition}

\bibliographystyle{amsplain} 
\bibliography{references.bib}

\end{document}